\documentclass[12pt,a4paper]{amsart}
\usepackage{amsmath, amssymb, amsthm, geometry, graphicx,todonotes,mathtools,overpic} 

\usepackage{hyperref}
\usepackage[capitalise, nameinlink]{cleveref}

\setuptodonotes{inline}

\newtheorem{theorem}{Theorem}[section]
\newtheorem{lemma}[theorem]{Lemma}
\newtheorem{proposition}[theorem]{Proposition}

\theoremstyle{definition}

\newtheorem{definition}[theorem]{Definition}

\newcommand{\NN}{\mathbb{N}}
\newcommand{\RR}{\mathbb{R}}
\newcommand{\R}{\RR}

\DeclareMathOperator*{\conv}{conv}

\newcommand{\cA}{\mathcal{A}}
\newcommand{\cC}{\mathcal{C}}

\newcommand{\nonneg}[1]{\NN^{#1}} 
\newcommand{\even}[1]{2 \NN^{#1}} 
\newcommand{\SONC}[1]{\cC({#1})} 
\newcommand{\Pspace}[1]{\R[{#1}]} 

\newcommand{\alp}{\alpha}

\newcommand{\Vector}[1]{\ensuremath{\boldsymbol{#1}}}

\newcommand{\alpb}{{\Vector{\alp}}}
\newcommand{\betab}{{\Vector{\beta}}}
\newcommand{\gammab}{{\Vector{\gamma}}}

\newcommand{\xb}{\mathbf{x}}

\definecolor{NiceBlue}{rgb}{0.2,0.2,0.75}
\newcommand{\struc}[1]{{\color{NiceBlue} #1}}

\author{Mareike Dressler}
\address{Mareike Dressler, School of Mathematics and Statistics, University of New South Wales, Sydney, NSW 2052, Australia.}

\author{Hongzhi Liao}
\address{Hongzhi Liao, School of Mathematics and Statistics, University of New South Wales, Sydney, NSW 2052, Australia.}

\author{Vera Roshchina}
\address{Vera Roshchina, School of Mathematics and Statistics, University of New South Wales, Sydney, NSW 2052, Australia.}



\title{Exposed extreme rays of the SONC cone}

\begin{document}

\begin{abstract}
We provide a complete and explicit characterization of the exposed extreme rays of the cone of sums of nonnegative circuit (SONC) polynomials.   
The criterion we derive is purely combinatorial and depends only on the existence of certain circuits within the ground set and on the nature of the corresponding extreme ray.
Our constructive proofs also yield explicit exposing functionals, offering a basis for algorithmic detection of exposed rays in SONC-based optimization.
\end{abstract}

\maketitle

\section{Introduction}

A central objective in real algebraic geometry, 
and a driving force in polynomial optimization,
is to understand the structure of the cone of nonnegative polynomials and, crucially, its \textit{tractable} subcones, that is, those for which membership can be efficiently verified.
The most prominent such subcone is the cone of sums of squares (SOS), whose relationship to nonnegativity has been extensively studied since Hilbert’s seminal 1888 work~\cite{Hilbert1888}.
Over the past decades, the facial geometry of the cone of nonnegative polynomials and the SOS cone has received sustained attention, including analyses of faces, extreme and exposed rays, the boundary, and rank structure, see e.g. \cite{Blekherman:NNSOS,Blekherman:etal:Boundaries,Blekherman:Iliman:Kubitzke,Blekherman:Parrilo:Thomas}.
These works have substantially advanced the understanding of the structure and interplay of these cones, yet a complete understanding is still lacking.

\smallskip
A more recent example of a tractable subcone, which is independent of the SOS cone, is the cone of sums of nonnegative circuit (SONC) polynomials. 
First formally introduced in~\cite{iliman2016amoebas} and extending the earlier notion of agiforms from~\cite{reznick1989forms}, the SONC cone
provides a sparsity-preserving and combinatorially governed alternative to SOS-based certificates. An equivalent construction in the signomial setting, known as the cone of sums of arithmetic-geometric exponentials (SAGE), is introduced in~\cite{ChandrasekaranShah_REPforSignomialOptimization}, and related but independently developed frameworks can be found in~\cite{Fidalgo:Kovacec,PanteaEtAL_GlobalInjectivityAndMultipleEquilibria}.   

Despite several recent works that have substantially advanced our knowledge of the geometry of the SONC cone, see e.g.~\cite{katthan2021unified,dressler2021real,AlgBound}, its facial structure remains far from fully understood. In particular, it is still unknown which of its faces are exposed.  In~\cite{dressler2021real}, the first author initiated the study of exposed faces of the SONC cone, hereby focusing on exposed faces arising from polynomials vanishing at a finite set of points in $\R^n$ and providing explicit classifications in low dimensions together with general dimension bounds.

\smallskip
In this work, we (partially) close this gap by fully characterizing the exposed extreme rays of the SONC cone. In contrast to the SOS cone and the cone of nonnegative polynomials, whose facial structure is amenable to the methods of algebraic geometry, the SONC cone is built as the conic hull of elementary cones, with structure governed by the combinatorics of the underlying polynomials. We use a convex geometric approach to characterize the exposed extreme rays of the SONC cone in terms of its combinatorial features.

A polynomial $f(\xb)=\sum_{\alpb \in S} c_{\alpb} \xb^{\alpb} + d \xb^{\betab}$, (where we use the standard shorthand for writing multivariate monomials, $\xb^\alpb \coloneqq x_1^{\alp_1} x_2^{\alp_2} \cdots x_n^{\alp_n}$), where $S$ is an affinely independent subset of $\even{n}$ with $\betab\in \nonneg{n}$ in the relative interior of the convex hull of $S$ and the coefficients $c_{\alpb}>0$, is called a circuit polynomial; see also Section~\ref{subsec:nncircuits}. Deciding  nonnegativity of a circuit polynomial is equivalent to solving a system of linear equations, a fact that follows directly from the arithmetic-geometric inequality. The distinguishing feature of this approach is its sparsity-preserving nature, making it particularly well-suited for handling high-degree polynomials with sparse support. This motivates us to consider the SONC cone on a given finite ground set $\cA\subseteq \nonneg{n}$ as the set of all SONC polynomials supported on $\cA$, that is, all nonnegative combinations of circuit polynomials with ``circuits" on $\cA$ (i.e., $S\subseteq \cA$ and $\betab\in \cA$). For the explicit definition, see Sections~\ref{subsec:nncircuits} and \ref{subsec:SONC}.

\medskip
Our main result is the following characterization of exposed rays of the SONC cone.

\begin{theorem}\label{thm:main} An extreme ray $r$ of the SONC cone on a finite ground set $\cA\subseteq\nonneg{n}$ is not exposed if and only if $r=\R_+ \xb^\alpb$ for some $\alpb\in \cA$, and there exists a circuit $(S,\betab)$ on $\cA$ such that $\alpb \in S$ and $\betab\neq \alpb$ is even.
\end{theorem}

This result is proved by showing separately that the extreme rays that correspond to the conditions in the theorem are not exposing, and constructing explicit exposing normals for all of the remaining cases. The construction of these normals relies on a graded partition of the ground set $\cA$ that allows precise control over the relative magnitudes of the normal entries and reflects the combinatorial layering inherent in $\cA$.

\smallskip
While it is conceivable that the existence of such exposing normals could be deduced from general convex analysis arguments (as discussed in Section~\ref{sec:open}), our proof is fully constructive in the sense that we provide explicit normals certifying exposure of each ray. 
This constructive approach not only clarifies the geometry of the SONC cone but also lends itself to algorithmic implementation, potentially contributing to practical optimization frameworks that incorporate SONC certificates.

\medskip
Our paper is organized as follows. Section~\ref{sec:preliminaries} sets up the notation and reviews key properties of the SONC cone, including the characterization of its extreme rays.  In Section~\ref{sec:unexposed} we establish the negative part of the result in Proposition~\ref{prop:unexposed}. In Section~\ref{sec:exposed} we show that the remaining extreme rays are exposed in Propositions~\ref{prop:monomial} and~\ref{prop:nonmonomial}, and finalize the proof of Theorem~\ref{thm:main}. We close with a brief discussion of open problems in Section~\ref{sec:open}.

\section{Preliminaries}\label{sec:preliminaries}
Throughout the article, we use $\struc{\NN}$ and $\struc{\R}$ to denote the sets of nonnegative integers and real numbers, respectively. $\struc{\R_+}$ indicates the nonnegative elements of $\R$. 
For a set $S\subseteq \NN^n$, we denote its \struc{\emph{convex hull}} by $\struc{\conv(S)}$ and the set of \struc{\emph{vertices}} (extreme points) of a convex set $\conv(S)$ by $\struc{V(S)}$. We call a lattice point $\alpb \in \NN^n$ \struc{\emph{even}} if every entry $\alp_i$ is even, i.e., $\alpb\in 2\NN^n$.

\subsection{Nonnegative circuit polynomials}\label{subsec:nncircuits}
Recall that a \struc{\emph{circuit}} is a set that is minimal affine dependent, see e.g. \cite{GKZ:Discriminants}. If a circuit $S\cup \{\betab\}\subseteq \NN^n$ is such that $S$ is affinely independent and $\betab$ belongs to the relative interior of $\conv(S)$, we call the circuit \emph{simplicial}. Since we are only dealing with simplicial circuits, we just refer to this case as a circuit hereafter. We will also use the notation \struc{$(S,\betab)$} to denote such circuits to streamline our exposition. 
Further, we say a circuit $(S,\betab)$ is \struc{\emph{odd}} or \struc{\emph{even}} according to the parity of $\betab$.  
Note that for a circuit, $\conv(S)$ is a simplex, and hence there is a unique way to represent $\betab$ as a convex combination of $S$. That is, there exists unique \struc{\emph{barycentric coordinates $(\lambda_{\alpb})_{\alpb \in S}$}} (with respect to $S$) such that $\betab = \sum_{\alpb \in S} \lambda_\alpb \alpb$, with $\lambda_\alpb>0$ for all $\alpb \in S$ and $\sum_{\alpb \in S} \lambda_\alpb = 1$. 

\smallskip
In what follows, we are interested in sparse polynomials supported on a finite set. Therefore, let $\struc{\cA}\subseteq \NN^n$ be a distinguished finite ground set and consider $\struc{\Pspace\cA}$,   the ring of real polynomials supported on $\cA$. 
A polynomial $f\in \Pspace\cA$ is called a \struc{\emph{circuit polynomial}} if it is supported on a circuit $(S,\betab)$ and $S\subseteq 2\NN^n$ or it is a (sum of) monomial square(s). 
This implies that we can write $f$ as
\begin{equation}\label{eq:defcircuit}
f(\xb)=\sum_{\alpb \in S} c_{\alpb} \xb^{\alpb} + d \xb^{\betab},
\end{equation}
with coefficients $c_\alpb >0$ for all ${\alpb\in S}$, $d\in \RR$, and support $S\subseteq 2\NN^n$, $\betab \in \NN^n$. 
The \struc{\emph{circuit number}} of $f$ is defined as 
\begin{equation}\label{eq:defcircuitnumber}
\Theta_f  \coloneqq \prod_{\alp \in S} \left( \frac{c_\alpb}{\lambda_\alpb}\right)^{\lambda_\alpb}.
\end{equation}
 
This invariant can be used to easily check if a circuit polynomial is nonnegative, see \cite[Theorem 1.1]{iliman2016amoebas}. Let $f$ be a circuit polynomial as in~\eqref{eq:defcircuit}, then $f$ is nonnegative if and only if 
 $f$ is a sum of monomial squares, or
 $   |d|\leq \Theta_f$.

\subsection{The SONC cone and its extreme rays}\label{subsec:SONC} 
If a polynomial can be written as a sum of nonnegative circuit polynomials, we call it a \struc{\emph{SONC polynomial}}.
The \struc{\emph{SONC cone $ \SONC{\cA}$}} on some finite ground set $\cA\subseteq \nonneg{n}$ is the set of all conic (finite nonnegative) combinations of nonnegative circuit polynomials in $\Pspace\cA$. It is not necessary to consider all possible circuits on $\cA$ to represent each polynomial in $\SONC{\cA}$ as a conic combination of nonnegative circuit polynomials; instead, the representation can be restricted to reduced circuits. We adjust the original definition of reduced circuits given in \cite{katthan2021unified}  to our less general setting as follows.

\begin{definition} Let $\cA\subseteq \nonneg{n}$ be finite. A circuit $(S,\betab)$ on $\cA$ is \struc{\emph{reduced}} (with respect to $\cA$) if 
\[
\conv (S) \cap \cA \cap \even{n} = S\cup \{\betab\} \cap \even{n}.
\]  
\end{definition}

In other words, a circuit is reduced if there are no even points in its convex hull apart from the points in $S\cup\{\betab\}$.

\smallskip
Recall that a ray $r=\R_+\cdot v$ (for some $v\neq 0$) is an \struc{\emph{extreme ray}} of a convex cone $K$ if for any $y,z\in K$ such that $x = y+z$, we have $y,z\in \R_+\cdot v$. In other words, an extreme ray is a one-dimensional face of a pointed cone $K$.

The concept of reduced circuits leads to a complete characterization of the extreme rays of the SONC cone.

\begin{proposition}[\cite{katthan2021unified}, Corollary~4.6]\label{prop:extreme} 
Let $r$ be an extreme ray of $\SONC{\cA}$. Then $r = \RR_+ f$, where $f\in \Pspace\cA$ belongs to one of the following three types of nonnegative circuit polynomials:
\begin{itemize}
    \item[1.] $f(\xb) = \xb^\betab, \quad \betab\in \cA\cap \even{n}$;
    \item[2.] $f$ is supported on some (even or odd) reduced circuit $(S,\betab)$, with $|S|>1$ and  
    \[
    f(\xb) = \sum_{\alpb \in S} c_\alpb \xb^\alpb - \Theta_f \xb^\betab,\quad c_\alpb >0\quad  \text{ for all } \; \alpb \in S;
    \]
    \item[3.] $f$ is supported on an odd reduced circuit $(S,\betab)$, with $|S|>1$ and  
    \[
    f(\xb) = \sum_{\alpb \in S} c_\alpb \xb^\alpb +  \Theta_f \xb^\betab,\quad c_\alpb >0\quad  \text{ for all } \; \alpb \in S.
    \]
\end{itemize}
\end{proposition}

\section{Unexposed rays}\label{sec:unexposed}

In this section we prove that the extreme rays of type 1 given in Proposition~\ref{prop:extreme} that satisfy the conditions of Theorem~\ref{thm:main} are unexposed. Together with the positive results of the next Section (showing that the remaining extreme rays are exposed) this proves Theorem~\ref{thm:main}.

Recall that a face $F$ of a convex cone $K\in X$ (where $X$ is a linear vector space) is \struc{\emph{exposed}} if there exists a linear function $\struc{l}\colon X\to \R$  such that
\begin{equation}\label{eq:defexposed}    
l(x)=0\; \, \text{ for all }\; x\in F,\quad \text{ and } \quad l(y)>0\;\, \text{ for all }\, y\in K\setminus F. 
\end{equation}

The following result contributes to the proof of the negative part of Theorem~\ref{thm:main}.  
\begin{proposition}\label{prop:unexposed} Let $\cA\subseteq \nonneg{n}$ be a finite set, and let $\gammab \in \cA\cap \even{n}$ be such that there exists a circuit $(S,\betab)$ on $\cA$ with $|S|>1$,  $\gammab\in S$, and $\betab$ even, then the extreme ray $r= \R_+ \xb^\gammab$ of $\SONC{\cA}$ is not exposed.
\end{proposition}
Before proving Proposition~\ref{prop:unexposed}, we consider the following illustrative example based on the classical case of a SONC cone associated with the Motzkin polynomial. The affine version of the Motzkin polynomial is 
\[
p_M (x,y) = 1+x^2 y^4 + x^4 y^2 - 3x^2 y^2,
\]
and it can be shown to be a nonnegative circuit polynomial using the SONC approach. Indeed, consider the SONC cone on the ground set that consists of the monomial degrees of $p_M$ (its Newton polytope is shown in Fig.~\ref{fig:Newton}).
\begin{figure}[ht]
    \centering
    \begin{overpic}[width=0.35\textwidth]{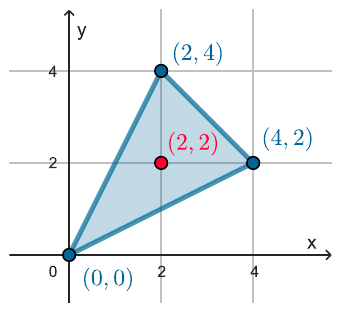} 
\end{overpic}
    \caption{Newton polytope of the Motzkin polynomial}
    \label{fig:Newton}
\end{figure} 
There is only one possible circuit on this ground set, $(S,\betab)$ with $S = \{(0,0), (2,4), (4,2)\}$ and $\betab = (2,2)$. Notice that for each $\gammab \in S$ the conditions of Proposition~\ref{prop:unexposed} are satisfied, so we expect that the rays $\R_+$, $\R_+ x^2 y^4$, and $\R_+ x^4 y^2$  are unexposed extreme rays of the SONC cone $\SONC{\cA}$ on the ground set $\cA = \{(0,0), (2,2), (2,4), (4,2)\}$. Since the cone $\SONC{\cA}$ is four-dimensional, we can visualize its compact three-dimensional slice, and witness that the corresponding extreme points are unexposed. This is illustrated in Fig.~\ref{fig:Motzkin},
\begin{figure}[ht]
    \centering
    \begin{overpic}[width=0.42\textwidth]{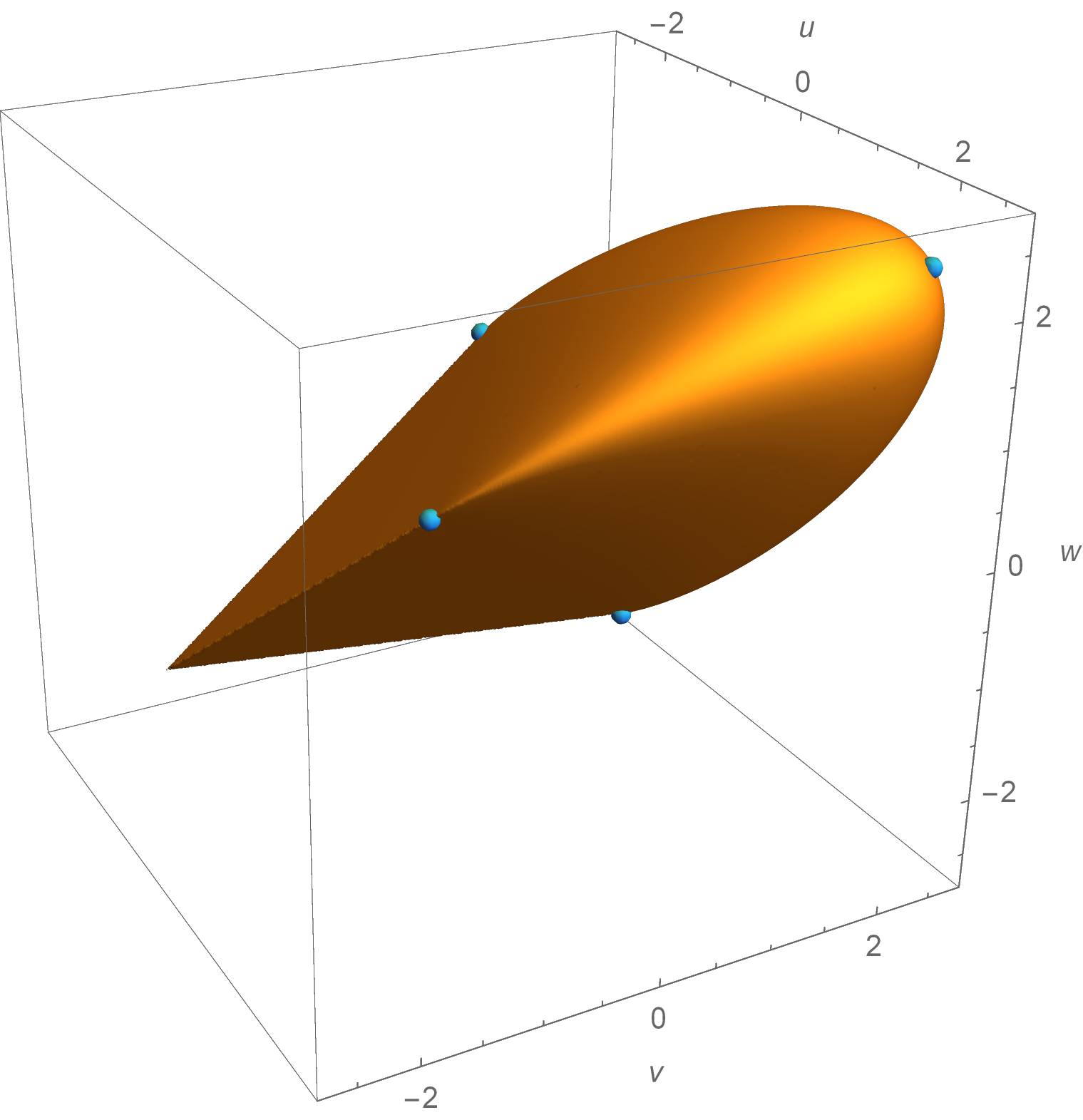} 
  \put(35,55){{\color{white}$c $}}
  \put(55,40){$c x^2 y^4 $}
  \put(29,72){$c x^4 y^2 $}  
  \put(85,77){$d p_M $}
\end{overpic}\qquad
    \caption{A three-dimensional affine slice of the four-dimensional SONC cone on the ground set of the Motzkin polynomial $p_M$. The coordinates $u$, $v$, and $w$ correspond to the representation of the slicing affine subspace as $L = p_0+ \mathrm{span}\ \{p_1,p_2,p_3\}$ with $p = p_0 + u p_1+ v p_2 + w p_3$ for each $p\in L$. Here  $p_0 = 2 + 2 x^2 y^4  + 2 x^4 y^2 + x^2 y^2$, $p_1 =  2 + x^2 y^4  - 2 x^4 y^2 - 2 x^2 y^2$, $p_2 = -2 + 2x^2 y^4  + x^4 y^2 - 2 x^2 y^2$, $p_3 =  1 - 2 x^2 y^4  + 2 x^4 y^2 - 2 x^2 y^2$. The constants are $c = 13/2$ and $d = 13/3$.}
    \label{fig:Motzkin}
\end{figure} 
where the corresponding extreme points are visible on a slice of $\SONC{\cA}$.
\begin{proof}[Proof of Proposition~\ref{prop:unexposed}]
Suppose to the contrary that under the assumptions of the proposition, the extreme ray $r=\R_+ \xb^\gammab$ is exposed. Let $(\lambda_\alpb)_{\alpb \in S}$ be the barycentric coordinates of $\betab$ with respect to $S$, and let
\[
f_t (\xb) \coloneqq \lambda_\gammab \xb^{\gammab}+ \sum_{\alpb \in S\setminus \{\gammab\}} \lambda_{\alpb} t^{\frac{1}{1-\lambda_{\gammab}}}\xb^\alpb  - t \xb^\betab.  
\]
Notice that for any $t>0$, the polynomial $f_t$ is a circuit polynomial on $(S,\betab)$ with $\Theta_{f_t} =t$, and therefore $f_t\in \SONC{\cA}$. 

By assumption, there exists some linear function $l\colon \Pspace{\cA}\to \R$ exposing the extreme ray $r$. Explicitly, this means that $l(\xb^\gammab) =0$ and $l(g)>0$ for any $g\in \SONC{\cA}\setminus r$. 

By the linearity of $l$, we have 
\begin{align*}
l(f_t) 
& = t \left(t^{\frac{\lambda_{\gammab}}{1-\lambda_{\gammab}}} \sum_{\alpb \in S\setminus \{\gammab\}} \lambda_{\alpb} l(\xb^\alpb)  -  l(\xb^\betab)\right).
\end{align*}
Because $l(\xb^\betab)>0$ ($\betab\neq \gammab$ since $|S|>1$), and the first term diminishes as $t$ goes to zero, we must have $l(f_t)<0$ for a sufficiently small $t>0$, which is a contradiction since $f_t\in \SONC{\cA}$ and thus $l(f_t)\geq 0$ for all $t>0$.
\end{proof}

\section{Exposed extreme rays}\label{sec:exposed}

Our aim in this section is to show that all extreme rays of the SONC cone that do not satisfy the assumptions of Proposition~\ref{prop:unexposed} are exposed. 
We prove this for the remaining monomial extreme rays in Proposition~\ref{prop:monomial}, and then handle the unsettled cases in Proposition~\ref{prop:nonmonomial}. 
Before stating and proving these results, we introduce several technical statements, including one about a graded partition of the ground set $\cA$, which we use in the explicit construction of exposing linear functionals.

\medskip
Recall that, just as a compact convex set is the convex hull of its extreme points (Krein-Milman theorem), any pointed closed convex cone is the conic hull of the generators of its extreme rays (see \cite[Corollary~18.5.2]{rockafellar}). Hence, to determine whether an extreme ray is exposed, it suffices to verify that equation~\eqref{eq:defexposed} holds for these generators, as formalized in the next lemma.

\begin{lemma}\label{lem:extremeChar} An extreme ray $r = \RR_+ f$ of $\SONC{\cA}$ on a finite ground set $\cA \neq \emptyset$ is exposed if there exists a linear mapping $l\colon \Pspace{\cA}\to \R$ such that 
\[
 l (f)= 0,
\]
and for any other extreme ray $q = \RR_+ g$ of $\SONC{\cA}$ 
\[
l(g)>0.
\]
\end{lemma}
\begin{proof} Let $\cA$ be a finite ground set, $f\in \SONC{\cA}$, and suppose $l\colon \Pspace{\cA}\to \R$ is such that $r = \R_+ f$ is an extreme ray of $\SONC{\cA}$, $l (f)= 0$, and for any extreme ray $q = \RR_+ g$ with $f\notin q$ we have $l(g) >0$. 

For any $f'\in r = \R_+ f$, it holds $f' = t f$ for some $t\in \R_+$; hence $l(f') = t l(f) =0$. To prove that $l$ exposes $r$ as a face of $\SONC{\cA}$, it remains to verify that for any $h\in \SONC{\cA}\setminus r$, we have $l(h)>0$. Since $\SONC{\cA}$ is a pointed convex cone (it is easy to observe that the SONC cone does not contain lines and thus is a pointed cone; cf Proposition~\ref{prop:extreme} or see also~\cite{dressler2017positivstellensatz}), by \cite[Corollary~18.5.2]{rockafellar} we have 
\[
h = \sum_{i=1}^p \mu_i g_i,
\]
where $g_i$ spans some extreme ray of $\SONC{\cA}$ and $\mu_i>0$  for each $i\in \{1,\dots, p\}$ (because $h\notin r$, we must have $h\neq 0$). Moreover, there is some $j \in \{1,\dots, p\} $ such that $g_j\notin r$, and hence $l(g_j)>0$, with $l(g_i)\geq 0$ for all remaining $i\in \{1,\dots, p\}$. We conclude that  
\[
l(h) = \sum_{i=1}^p \mu_i l(g_i) \geq \mu_j l(g_j) >0.
\]
\end{proof}

Next, we describe a \emph{graded partition} of a finite set that we utilize in the proofs of the remaining cases.

\begin{lemma}\label{lem:graded} For a finite set $\cA\subseteq \nonneg{n}$ and a reduced circuit $(S,\betab)$ on $\cA$, there exists a \struc{\emph{graded partition}} $\struc{L_0}, \struc{L_1},\dots, \struc{L_K}$ of $(\cA\cap \even{n})\cup \{\betab\}$ such that $L_0 = S\cup \{\betab\}$,
\begin{itemize}
    \item[(i)]$L_0\cup L_1\cup \cdots \cup L_K = (\cA\cap \even{n})\cup \{\betab\}$;\quad $L_i\cap L_j = \emptyset$ $\text{ for all } i\neq j$; and
    \item[(ii)] for any circuit $(S',\gammab)\neq (S,\betab)$, $|S'|>1$ on $\cA$ with $\gammab\in (\cA\cap \even{n})\cup \{\betab\}$ there exists $\alpb \in S'$ such that $\alpb\in L_j$, $\gammab\in L_i$ with $j>i$.
\end{itemize}
\end{lemma}

\begin{proof}

Let $\cA\subseteq \nonneg{n}$ be a finite set, and let $(S,\betab)$ be a reduced circuit on $\cA$. Let $L_0 \coloneqq S\cup \{\betab\}$, $E_0 \coloneqq \cA\cap \even{n}\cup \{\betab\}$, and define
\[
E_{i+1} \coloneqq E_{i}\setminus (V(E_{i})\setminus S)\quad \text{ for all } \, i\in \nonneg{}.
\]
In other words, $E_{i+1}$ is constructed by removing the extreme points of $\conv(E_i)$, except for the ones that belong to $L_0$ (see Fig.~\ref{fig:graded} for an illustrative example). 
\begin{figure}[ht]
    \centering
    \begin{overpic}[width=0.45\textwidth]{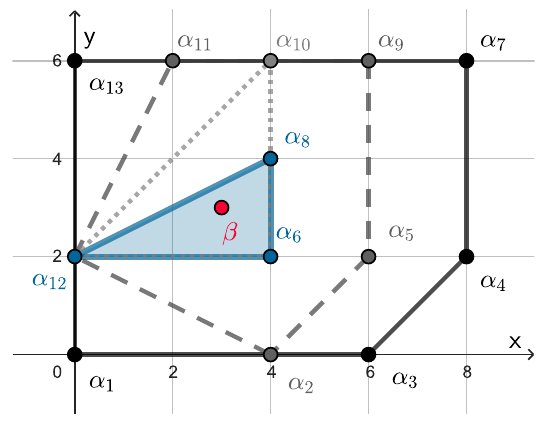} 
  \put(87,58){$E_0$}
  \put(70,58){$E_1$}
  \put(52,58){$E_2$}
  \put(27,25){$E_3=L_0$}
\end{overpic}
    \begin{overpic}[width=0.45\textwidth]{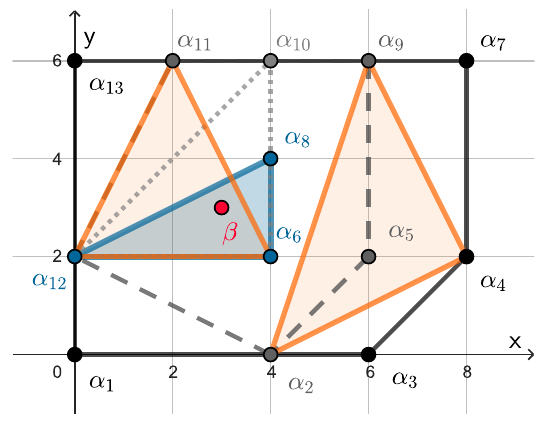} 
  \put(87,58){$L_3$}
  \put(75,58){$L_2$}
  \put(52,58){$L_1$}
  \put(32,25){$L_0$}
\end{overpic}
    \caption{An example of a graded partition from the proof of Lemma~\ref{lem:graded} for $\cA = \{\alpb_1,\dots,\alpb_{13},\betab\}$ and $S = \{\alpb_6,\alpb_8,\alpb_{12}\}$. The sets $E_0$, $E_1$, $E_2$, and $E_3$ are shown as the outlines of the boundaries of their convex hulls. For the circuit $S'=(\{\alpb_6,\alpb_{11},\alpb_{12}\},\betab)$ we have $\gammab  =\betab\in L_0$ and $\alpb \in S'\setminus S$ with $\alpb 
    = \alpb_{11}\in L_2$, while for $(\{\alpb_2,\alpb_4,\alpb_{9}\},\alpb_5)$ we have $\gammab = \alpb_5 \in L_2$ and $\alpb = \alpb_4\in L_3$.}
    \label{fig:graded}
\end{figure} 

For $i\in \nonneg{}$, we have $E_{i+1}=E_{i}$ if and only if $E_i = S \cup \{\betab\}=L_0$. Furthermore, if $L_0\subsetneq E_i$, then $L_0\subsetneq V(E_i)$, and $E_i\subsetneq E_{i+1}$. We conclude that there exists some $K\in \nonneg{}$ such that 
\[
 E_K = L_0; \quad E_{i}\subsetneq E_{i-1} \quad \text{ for all } i \in \{1,\dots, K\}.
 \]
 Moreover, for every $\alpb \in  E_0 \setminus L_0$ there must exist some $i\in \{1,\dots, K\}$ such that $\alpb \in E_{i-1}\setminus E_{i}$. 
We set 
\[
L_{K-i+1} \coloneqq  E_{i-1}\setminus E_{i} \quad \text{ for all } i\in \{1,\dots, K\}.
\]
To see that $L_0,L_1,\dots, L_K$ satisfy (i), observe that by construction these sets are disjoint and 
\[
L_1\cup \cdots \cup L_K = (E_{K-1}\setminus E_{K})\cup \cdots \cup (E_{i-1}\setminus E_{i})\cup \cdots (E_{0}\setminus E_{1}) = E_0\setminus E_K = E_0\setminus L_0,
\]
therefore 
\[
L_0\cup L_1\cup \cdots \cup L_K = E_0 = \cA\cap \even{n}\cup \{\betab\}.
\]
\smallskip

It remains to prove (ii). The subsequent case distinction is illustrated on the right-hand side of Fig.~\ref{fig:graded}.

Suppose that $(S',\gammab)$ is a circuit on $\cA$, such that $(S',\gammab)\neq (S,\betab)$ and $|S'|>1$. If $\gammab = \betab$, then we must have some $\alpb \in S'\setminus S$ (otherwise $S'\subset S$ and $(S,\betab)$ cannot be a circuit). Since $\alpb \in E_0$, there must be some $i$ such that $\alpb \in L_i$ with $i>0$ because $\alpb\notin L_0= S\cup \{\betab\}$
and we have $\gammab \in L_0$.

In the case when $\gammab \neq \betab$, we have $\gammab\in \cA\cap \even{n}$, and hence there is some $i\in \{0,\dots, K\}$ such that $\gammab \in L_i$. We need to show that there is $\alpb \in S'$ such that $\alpb\in L_j$ with $j>i$. Suppose that this is not true. Then for every $\alpb\in S'$, we must have $\alpb \in L_{j}$ with $j\leq i$. By our construction, it means that either $\gammab \in S$ (in which case $\gammab\in L_0$ and $S'\subset L_0$, an impossibility) or $\gammab \in V(E_{i})\setminus S$, and $S'\subset E_i$. Since $\gammab$ is a vertex of $\conv(E_i)$, the only possible way to represent $\gammab$ as a convex combination of points in $E_i$ (including $S$) is $\gammab = 1\cdot \gammab$, which contradicts $|S'|>1$.
\end{proof}

The following result specifies a particular positive integer that is subsequently employed in the definition of the exposing functionals. 
\begin{lemma}\label{lem:Lambda}
Given a nonempty finite set $\cA\subseteq\nonneg{n}$, there exists an integer $\struc{\Lambda>0}$ such that for any circuit $(S,\betab)$ on $\cA$, we have $\lambda_\alp \cdot \Lambda >1$ for all $\alpb \in S$, with $(\lambda_\alpb)_{\alpb\in S}$ being the barycentric coordinates of $\betab$.
\end{lemma}

\begin{proof}
Let $C(\cA)$ be the set of all circuits on $\cA$, and for every circuit $(S,\betab)\in C(\cA)$, let $\lambda(S,\betab)$ be the smallest barycentric coordinate of $\betab$ with respect to $S$. Note that $\lambda(S,\betab)>0$.  The number of circuits $|C(\cA)|$ is finite due to the finiteness of $\cA$, thus
\[
\lambda \coloneqq \inf_{(S,\betab)\in C(\cA)}\lambda(S,\betab) = \min_{(S,\betab)\in C(\cA)}\lambda(S,\betab) >0.
\]
Let $\struc{\Lambda} \coloneqq \lambda^{-1}+1$. By construction, for any circuit $(S,\betab)$ on $\cA$ and for any barycentric coordinate $\lambda_\alpb$ of $\betab$ (with respect to $S$), we have 
\[
\Lambda\cdot \lambda_\alp \geq \Lambda\cdot \lambda(S,\beta) \geq \Lambda \cdot \lambda = (\lambda^{-1}+1)\lambda = 1 + \lambda>1.
\]
\end{proof}

Recall the \emph{\struc{weighted AM-GM inequality}}, see e.g.~\cite[Section 2.5]{Hardy:Littlewood:Polya}.
\begin{proposition}[weighted AM-GM inequality]\label{prop:AMGM}
    Let $a_1, a_2, \dots, a_m$ be nonnegative real numbers, and suppose $w_1, w_2, \dots, w_m$ are nonnegative real numbers such that $w_1 + w_2 + \dots + w_m = 1$. Then
\[
w_1 a_1 + w_2 a_2 + \dots + w_m a_m \geq a_1^{w_1} a_2^{w_2} \dots a_m^{w_m},
\]
with equality if and only if $a_1 = a_2 = \dots = a_m$ and $w_i > 0$ for all $i \in \{1,\dots, m\}$.
\end{proposition}

\begin{proposition}\label{prop:monomial}
    If $r = \R_+ \xb^{\gammab}$ is an extreme ray of $\SONC{\cA}$ on a finite set $\cA\subseteq \nonneg{n}$, and there is no circuit $(S,\betab)$ on $\cA$ such that $|S|>1$,  $\gammab\in S$, and $\betab$ even, then $r$ is exposed.
\end{proposition}
\begin{proof} Let $\xb^\gammab$ and $\cA$ satisfy the assumptions of the proposition. We prove that the extreme ray $r = \R_+ \xb^\gammab$ is exposed by constructing an explicit exposing linear mapping $l\colon \Pspace{\cA}\to \R$.

By Lemma~\ref{lem:graded}, there exists a graded partition $L_0,L_1,L_2,\dots, L_K$ of $\cA\cap \even{n}\cup \{\gammab\}$ such that $L_0   = \{\gammab\}$, and for any circuit $(S,\betab)$ with $\betab\in \cA\cap \even{n}\cup \{\gammab\}$, we have $\betab\in L_i$, and some $\bar \alpb\in S$ is such that $\bar \alpb\in L_j$ with $j>i$. We define
\begin{equation}\label{eq:defNormalMonomial}
l(\xb^\alpb) \coloneqq \begin{cases}
    2^{\Lambda^i}, & \text{ if } \alpb \in L_i \quad \text{for some } i\in \{1,\dots, K\},\\
    0, & \text{otherwise},
\end{cases}
\end{equation}
where $\Lambda$ is chosen as in Lemma~\ref{lem:Lambda}. To demonstrate that the linear mapping $l$ is exposing the extreme ray $r$, by Lemma~\ref{lem:extremeChar}, it is sufficient to show that $l(\xb^\gammab)=0$ (which is true by definition of $l$) and that for any other extreme ray $q=\R_+f$ of $\SONC{\cA}$, we have $l(f)>0$. We need to verify this inequality for all extreme rays listed in Proposition~\ref{prop:extreme}.

For any other monomial extreme ray $q = \R_+ \xb^\alpb$, where $\alpb \in (\cA\setminus \{\gammab\})\cap \even{n}$, there is an $i$ with $\alpb \in L_i$; hence, by \eqref{eq:defNormalMonomial}, we have $l(\xb^\alpb)>0$.

For any odd non-monomial extreme ray $q = \R_+ f$, where $f$ is a circuit polynomial on the odd circuit $(S,\betab)$, we have  
\[
f(\xb) = \sum_{\alpb \in S} c_\alpb 
\xb^\alpb \pm \Theta_f \xb^\betab,
\]
and since $l(\xb^\betab) = 0$, 
\[
l(f) = \sum_{\alpb \in S} c_\alpb l(\xb^\alpb) >0.
\]    

The only remaining type of extreme ray is $q = \R_+ f$, where 
\begin{equation}\label{eq:34543}
f(\xb) = \sum_{\alpb \in S} c_\alpb 
\xb^\alpb - \Theta_f \xb^\betab,
\end{equation}
and $(S,\betab)$ is an even circuit. Let $(\lambda_\alpb)_{\alpb\in S}$ be the barycentric coordinates of $\betab$. Then, by Proposition~\ref{prop:AMGM} we have 
\begin{align}\label{eq:44535}
    \sum_{\alpb \in S} c_\alpb l(\xb^\alpb) = \sum_{\alpb \in S} \lambda_\alpb \frac{c_\alpb}{\lambda_\alpb } l(\xb^\alpb) \geq \prod_{\alpb \in S} \left( \frac{c_\alpb l(\xb^\alpb) }{\lambda_\alpb } \right)^{\lambda_\alpb} = \Theta_f\prod_{\alpb \in S} \left( l(\xb^\alpb)\right)^{\lambda_\alpb}.
\end{align}
From \eqref{eq:34543}, \eqref{eq:44535}, and the nonnegativity of $l$ on extreme rays, we have
\begin{equation}\label{eq:9070}
l(f) \geq \Theta_f\left(\prod_{\alpb \in S} \left( l(\xb^\alpb)\right)^{\lambda_\alpb}- l(\xb^\betab)\right).    
\end{equation}

For every even $\alpb$, there is some $k$ such that $\alpb \in L_k$, and therefore
\begin{equation}\label{eq:9071}
\left( l(\xb^\alpb)\right)^{\lambda_\alpb} 
= \left(2^{\Lambda^{k}}\right)^{\lambda_\alpb} 
> 1 \quad \text{ for all } \alpb \in S.
\end{equation}
Moreover, by our construction, there is some $\bar \alpb\in S$ with $\betab \in L_i$, $\alpb\in L_j$, and $i<j$. Recalling that $|S|>1$, together with \eqref{eq:9070} and  \eqref{eq:9071} this yields  
\[
l(f) > \Theta_f\left(\left( l(\xb^{\bar\alpb})\right)^{\lambda_{\bar\alpb}}- l(\xb^\betab)\right)> \Theta_f(2^{\Lambda^{j-1}}-2^{\Lambda^i}) \geq 0.
\]
\end{proof}

\begin{proposition}\label{prop:nonmonomial} Every non-monomial extreme ray of the SONC cone on a finite ground set $\cA$ is exposed.
\end{proposition}

\begin{proof} Let $(S,\betab)$ be a reduced circuit with $|S|>1$, and let $f$ be a circuit polynomial on this circuit, generating one of the non-monomial extreme rays (that is, one of the extreme rays given in Proposition~\ref{prop:extreme} items 2 and 3). We have 
\[
f(\xb) = \sum_{\alpb \in S} c_\alpb 
\xb^\alpb \pm \Theta_f \xb^\betab,
\]
with $c_\alpb>0$ for all $\alpb \in S$. To show that $r = \R_+ f$ is an exposed ray, we construct the exposing linear mapping $l$ explicitly using a similar idea to the proof of Proposition~\ref{prop:monomial}. Choose some $\sigma, \delta>0$ such that 
\[
\sigma < \min\left \{\min_{\alpb\in S} \frac{\lambda_\alpb}{c_\alpb}, \Theta^{-1}_f\right\},\quad 
\delta> \max \left\{\max_{\alpb\in S}  \frac{\lambda_\alpb}{c_\alpb}, \Theta^{-1}_f
\right\}.\]
By Lemma~\ref{lem:graded}, there exists a graded partition $L_0,L_1,\dots, L_K$ of $(\cA\cap \even{n})\cup \{\betab\} $ such that $L_0 = S\cup \{\betab\}$ and for any circuit $(S',\gammab)$, where $\gammab\in (\cA\cap \even{n})\cup \{\betab\}$, there is an $\alpb'\in S'$ such that $\gammab\in L_i$, $\alpb'\in L_j$ with $j>i$. 
Let 
\begin{equation}\label{eq:constrNormals34}
l(\xb^\alpb) \coloneqq  \begin{cases}
    \sigma^{-1} \frac{\lambda_\alpb}{c_\alpb}, & \text{ if } \alpb \in S;\\
    \mp \sigma^{-1}\Theta^{-1}_f, & \text{ if } \alpb = \betab;\\
    (\sigma^{-1}\delta)^{\Lambda^i}, & \text{ if } \alpb \in L_i, \; i \in \{1,\dots, K\};\\    
    0, & \text{otherwise (when } \alpb \in \cA\setminus (\even{n}\cup \{\betab\}) \text{)}.
\end{cases}
\end{equation}
In the case $\alpb=\betab$, we specifically want the value $l(\xb^\betab)$ to have the opposite sign to the coefficient of $f$ at $\xb^\betab$. In particular, this means that the only case when $l(\xb^\gammab)$ is negative is when $\gammab = \betab$ and $\betab$ is odd (see Proposition~\ref{prop:extreme}, items 2 and 3). Observe that due to our choice of $\sigma$ and $\delta$, we have 
\[
|l(\xb^\gammab)|> 1 \quad  \text{ for all } \, \gammab\in (\cA\cap \even{n})\cup \{\betab\}.
\]
To prove that $r = \R_+ f$ is an exposed ray of $\SONC{\cA}$, by Lemma~\ref{lem:extremeChar} it is sufficient to show that for any extreme ray $q = \R_+ g$ of $\SONC{\cA}$, we have $l(g)>0$, unless $g\in \R_{+}f$, in which case we must have $l(g) =0$. We verify this by systematically considering all possible extreme rays of $\SONC{\cA}$ listed in Proposition~\ref{prop:extreme}.

\smallskip
The case $q = \R_+ \xb^\gammab$ with $\gammab\in \even{n}$ is trivial: by \eqref{eq:constrNormals34} we have $l(\xb^\gammab)>0$. 

\smallskip

It remains to consider all type 2 and 3 extreme rays given in Proposition~\ref{prop:extreme}, that is, 
\[
g(\xb) = \sum_{\alpb \in S'} c'_\alpb \xb^\alpb \pm \Theta_g \xb^\gammab,
\]
where $(S',\gammab)$ is a reduced circuit on $\cA$ with $|S'|>1$.

\smallskip

We consider the cases $(S',\gammab) = (S,\betab)$ and $(S',\gammab)\neq (S,\betab)$ separately.

\smallskip

Firstly, when $(S',\gammab) = (S,\betab)$, we have 
\[
g(\xb) = \sum_{\alpb \in S} c'_\alpb \xb^\alpb \pm \Theta_g \xb^\betab.
\]
If the signs of the coefficient at the monomial $\xb^\betab$ are opposite in $g$ and $f$, then we immediately have $l(g)>0$. If the sign is the same, then from our definition of $l$ in \eqref{eq:constrNormals34} 
\[
l(g) = \sigma^{-1}\left(\sum_{\alpb \in S} \lambda_\alpb \frac{c'_\alpb}{c_\alpb}  - \prod_{\alpb\in S} \left(\frac{c'_\alpb}{c_\alpb}  \right)^{\lambda_\alpb}\right).
\]
By Proposition~\ref{prop:AMGM} the expression inside the brackets is zero if and only if the coefficients $(c'_\alpb)_{\alpb \in S}$ are a (positive) multiple of the coefficients $(c_\alpb)_{\alpb \in S}$, meaning that $l(g)>0$ if and only if $g$ generates a different extreme ray than $f$. 

\smallskip

The only remaining case to consider is when $g$ is a polynomial on some reduced circuit $(S',\gammab)\neq (S,\betab)$, $|S'|>1$. We have 
\[
g(\xb) = \sum_{\alpb \in S'} c'_\alpb \xb^\alpb \pm \Theta_g \xb^\gammab,
\]
and 
\[
l(g) = \sum_{\alpb \in S'} c'_\alpb l(\xb^\alpb) \pm \Theta_g l(\xb^\gammab).
\]
If $\gammab \in \cA\setminus (\even{n}\cup \{\betab\})$, by construction \eqref{eq:constrNormals34}, we have $l(\xb^\gammab)=0$, and therefore $l(g)>0$. 
If $\gammab \in \cA\cap \even{n}\cup \{\betab\}$, then applying Proposition~\ref{prop:AMGM} in the same way as we did to obtain \eqref{eq:9070}, we have 
\[
l(g) \geq \Theta_g\left(\prod_{\alpb \in S'} \left( l(\xb^\alpb)\right)^{\lambda_\alpb}- |l(\xb^\gammab)|\right).
\]
For every even $\alpb$ we have $l(\xb^\alpb)\geq 1$, and by our construction of the graded layers $L_0,L_1,\dots,L_K$, there is some $\bar \alpb \in S'$ such that $\betab \in L_i$ and $\bar \alpb \in L_j$ with $j>i$. By our choice of $l$ in~\eqref{eq:constrNormals34} we have 
\[
l(g) \geq \Theta_g\left( l(\xb^{\bar \alpb})^{\lambda_{\bar \alpb}}- l(\xb^\gammab)\right)\geq \Theta_g\left((\sigma^{-1}\delta)^{\Lambda^j\lambda_{\bar \alpb}} -(\sigma^{-1}\delta)^{\Lambda^i} \right)>0,
\]
since $\Lambda \lambda_\alpb>1$ for all $\alpb$, and hence $\Lambda^j\lambda_{\bar \alpb}>\Lambda^{j-1}\geq \Lambda^i$, and $\sigma^{-1}\delta>1$ by definition. 
\end{proof}

Putting the results of Sections~\ref{sec:unexposed} and~\ref{sec:exposed} together yields our main result. 
\begin{proof}[Proof of Theorem~\ref{thm:main}] Propositions~\ref{prop:unexposed}, \ref{prop:monomial}, and~\ref{prop:nonmonomial} cover all extreme rays of the SONC cone (as listed in Proposition~\ref{prop:extreme}). The only unexposed case comes from Proposition~\ref{prop:unexposed}, which corresponds to the case stated in the theorem.
\end{proof}

\section{Conclusion and open questions}\label{sec:open}

In this work, we have characterized the exposed rays of the SONC cone on an arbitrary ground set using an explicit construction based on a graded partition of the ground set. Note that this characterization is purely combinatorial and is based on the positions of the monomial powers with respect to each other, while our proof is constructive and allows us to obtain an explicit description of exposing normals for each exposed ray of the SONC cone. 

A natural question is whether this approach can be extended to describe the exposed faces of SONC cones (not only rays) and whether a similar combinatorial description can be achieved in that setting. In the case of the extreme rays considered in this work, it was possible to glean geometric intuition and to formulate the appropriate algebraic conjectures from studying four- and three-dimensional examples of SONC cones that can be easily visualized, however it is more difficult to use this approach for general faces, where the corresponding SONC cone may need to be higher dimensional to exhibit a sufficiently sophisticated geometric structure.

\section*{Acknowledgments}
The authors are grateful to the Australian Research Council for the continuing support. This work was partially supported by the ARC Discovery Project DP200100124.
M.D. is supported by the ARC Discovery Early Career Award DE240100674.

\bibliographystyle{alpha}
\bibliography{refs}

\newcommand{\etalchar}[1]{$^{#1}$}
\begin{thebibliography}{BHO{\etalchar{+}}12}

\bibitem[BHO{\etalchar{+}}12]{Blekherman:etal:Boundaries}
G.~Blekherman, J.~Hauenstein, J.~C. Ottem, K.~Ranestad, and B.~Sturmfels.
\newblock Algebraic boundaries of {H}ilbert's {SOS} cones.
\newblock {\em Compos. Math.}, 148(6):1717--1735, 2012.

\bibitem[BIK15]{Blekherman:Iliman:Kubitzke}
G.~Blekherman, S.~Iliman, and M.~Kubitzke.
\newblock Dimensional differences between faces of the cones of nonnegative
  polynomials and sums of squares.
\newblock {\em Int. Math. Res. Not. IMRN}, (18):8437--8470, 2015.

\bibitem[Ble12]{Blekherman:NNSOS}
G.~Blekherman.
\newblock Nonnegative polynomials and sums of squares.
\newblock {\em J. Amer. Math. Soc.}, 25(3):617--635, 2012.

\bibitem[BPT13]{Blekherman:Parrilo:Thomas}
G.~Blekherman, P.A. Parrilo, and R.R. Thomas.
\newblock {\em {S}emidefinite {O}ptimization and {C}onvex {A}lgebraic
  {G}eometry}, volume~13 of {\em MOS-SIAM Series on Optimization}.
\newblock SIAM and the Mathematical Optimization Society, Philadelphia, 2013.

\bibitem[CS16]{ChandrasekaranShah_REPforSignomialOptimization}
V.~Chandrasekaran and P.~Shah.
\newblock Relative entropy relaxations for signomial optimization.
\newblock {\em SIAM J. Optim.}, 26(2):1147--1173, 2016.

\bibitem[DIdW17]{dressler2017positivstellensatz}
M.~Dressler, S.~Iliman, and T.~de~Wolff.
\newblock A {P}ositivstellensatz for {S}ums of {N}onnegative {C}ircuit
  {P}olynomials.
\newblock {\em SIAM J. Appl. Algebra Geom.}, 1(1):536--555, 2017.

\bibitem[Dre21]{dressler2021real}
M.~Dressler.
\newblock Real zeros of {SONC} polynomials.
\newblock {\em J. Pure Appl. Algebra}, 225(7):106602, 2021.

\bibitem[FdW22]{AlgBound}
J.~{Forsg{\r{a}}rd} and T.~de~Wolff.
\newblock The algebraic boundary of the sonc-cone.
\newblock {\em SIAM J. Appl. Algebra Geom.}, 6(3):468--502, 2022.

\bibitem[FK11]{Fidalgo:Kovacec}
C.~Fidalgo and A.~Kovacec.
\newblock Positive semidefinite diagonal minus tail forms are sums of squares.
\newblock {\em Math. Z.}, 269(3-4):629--645, 2011.

\bibitem[GKZ94]{GKZ:Discriminants}
I.~M. Gelfand, M.~M. Kapranov, and A.~V. Zelevinsky.
\newblock {\em Discriminants, resultants, and multidimensional determinants}.
\newblock Mathematics: Theory \& Applications. Birkh\"auser Boston, Inc.,
  Boston, MA, 1994.

\bibitem[Hil88]{Hilbert1888}
D.~Hilbert.
\newblock {\"Uber die Darstellung definiter Formen als Summe von
  Formenquadraten}.
\newblock {\em Math. Ann.}, 32:342–350, 1888.

\bibitem[HLP34]{Hardy:Littlewood:Polya}
G.~H. Hardy, J.~E. Littlewood, and G.~P\'{o}lya.
\newblock {\em Inequalities}.
\newblock Cambridge University Press, 1934.

\bibitem[IdW16]{iliman2016amoebas}
S.~Iliman and T.~de~Wolff.
\newblock Amoebas, nonnegative polynomials and sums of squares supported on
  circuits.
\newblock {\em Res. Math. Sci.}, 3:1--35, 2016.

\bibitem[KNT21]{katthan2021unified}
L.~Katth{\"a}n, H.~Naumann, and T.~Theobald.
\newblock A unified framework of {SAGE} and {SONC} polynomials and its duality
  theory.
\newblock {\em Math. Comp.}, 90(329):1297--1322, 2021.

\bibitem[PKC12]{PanteaEtAL_GlobalInjectivityAndMultipleEquilibria}
C.~Pantea, H.~Koeppl, and G.~Craciun.
\newblock Global injectivity and multiple equilibria in uni- and bi-molecular
  reaction networks.
\newblock {\em Discrete Continuous Dyn. Syst. Ser. B}, 17(6):2153--2170, 2012.

\bibitem[Rez89]{reznick1989forms}
B.~Reznick.
\newblock Forms derived from the arithmetic-geometric inequality.
\newblock {\em Math. Ann.}, 283(3):431--464, 1989.

\bibitem[Roc97]{rockafellar}
R.~T. Rockafellar.
\newblock {\em Convex analysis}.
\newblock Princeton Landmarks in Mathematics. Princeton University Press,
  Princeton, NJ, 1997.
\newblock Reprint of the 1970 original, Princeton Paperbacks.

\end{thebibliography}

\end{document}